\documentclass[runningheads]{llncs}
\usepackage[T1]{fontenc}
\usepackage{amssymb, mathtools}
\usepackage[dvipsnames]{xcolor}
\usepackage{comment}
\usepackage{multicol}
\usepackage[colorlinks=true, urlcolor=Blue, citecolor=Green, linkcolor=Blue]{hyperref}
\usepackage[nameinlink]{cleveref}
\usepackage{graphicx}
\begin{document}
\title{On the approximation of the Riemannian barycenter}
\author{Simon Mataigne\inst{1}\orcidID{0009-0004-1970-8102} \and P.-A.~Absil\inst{1}\orcidID{0000-0003-2946-4178} \and Nina Miolane\inst{2}\orcidID{0000-0002-1200-9024}}
\authorrunning{Simon Mataigne, P.-A.~Absil and Nina Miolane}
\institute{ICTEAM Institute, UCLouvain, Louvain-la-Neuve, Belgium.\\% \and Springer Heidelberg, Tiergartenstr. 17, 69121 Heidelberg, Germany
\email{simon.mataigne@uclouvain.be, pa.absil@uclouvain.be} \and
UC Santa Barbara, Electrical and Computer Engineering, Santa Barbara, CA.\\% \and Springer Heidelberg, Tiergartenstr. 17, 69121 Heidelberg, Germany
\email{ninamiolane@ucsb.edu}
}
\maketitle              % typeset the header of the contribution
\begin{abstract}
%The abstract should briefly summarize the contents of the paper in 150--250 words.
We present a method for computing an approximate Riemannian barycenter of a collection of points lying on a Riemannian manifold. Our approach relies on the use of theoretically proven under- and over-approximations of the Riemannian distance function. We compare it to Riemannian steepest descent on the exact objective function of the Riemannian barycenter and to an approach that approximates the Riemannian logarithm using lifting maps. Experiments are conducted on the Stiefel manifold.
\end{abstract}
\section{Introduction}
Computing the Riemannian barycenter of a dataset lying on a Riemannian manifold is a fundamental problem in statistics on manifolds. See,~e.g.,~\cite{Pennec2006} and references therein for applications. We outline the key aspects of this problem. 

Let $\mathcal{M}$ be a complete manifold equipped with a Riemannian metric $g$ and $d_g:\mathcal{M}^2\rightarrow \mathbb{R}$ be the distance function induced by the Riemannian metric; see \cref{sec:preliminaries} for details.
Given a collection of points $\{x_i\}_{i=1}^N\subset \mathcal{M}$ and a parameter $\nu\geq 1$, let
\begin{align}
\label{eq:karcher_mean}
f_\nu:\mathcal{M}\rightarrow\mathbb{R}:x\mapsto\left(\sum_{i=1}^N d_g(x,x_i)^\nu\right)^\frac{1}{\nu}\quad\text{and}\quad x^* \in \mathrm{arg}\hspace{-2pt}\min_{x\in\mathcal{M}}f_\nu(x).
\end{align}
Conditions regarding the existence and uniqueness of $x^*$ can be found, e.g., in~\cite{Afsari11}. The point $x^*$ is called a \emph{Riemannian $l^\nu$-barycenter} of $\{x_i\}_{i=1}^N$. Classical choices for $\nu$ are $\nu=1, 2$, or $\nu\rightarrow\infty$ (minimax problem). For $\nu=2$, $x^*$ generalizes the notion of Euclidean mean on Riemannian manifolds; it is then often referred to as a Fr\'echet mean or a Karcher mean \cite{karcher2014}. For $\nu=1$, $x^*$ generalizes the median. 

Solving~\eqref{eq:karcher_mean} usually relies on gradient-based optimization algorithms, which require the computation of \emph{Riemannian logarithms} at each iteration \cite{Afsari13}. %However, such computations are prohibitive on many manifolds such as the Stiefel manifold of orthonormal $p$-frames in $\mathbb{R}^n$, $n> p$, defined by,
However, there are important manifolds lacking tractable algorithms that are guaranteed to compute the Riemannian logarithm within prescribed accuracy, e.g., the Stiefel manifold of orthonormal $p$-frames in $\mathbb{R}^n$, $n> p$, defined by
\begin{equation*}
	\mathrm{St}(n,p)\coloneq\{X\in\mathbb{R}^{n\times p} \ | \ X^\top X =I_p\},
\end{equation*}
where $I_p$ is the $p\times p$ identity matrix. See \cite{AbsMahSep2008,ChakrabortyVemuri2019,EdelmanAriasSmith} and related works for applications of the Stiefel manifold. To tackle this computational bottleneck of \eqref{eq:karcher_mean}, an approach, described in \cref{sec:inverse_retraction}, approximates the Riemannian logarithm by \emph{lifting maps} \cite{Tetsuya2013,bouchard2025}. 

Our approach replaces the exact objective function $f_\nu$ in \eqref{eq:karcher_mean} by an approximate objective function $h_\nu$ that is based on theoretically guaranteed under- and over-approximations of $f_\nu$. The estimated Riemannian $l^\nu$-barycenter is then obtained by computing
\begin{equation}\label{eq:approx_mean}
	\widehat{x}^* \in \mathrm{arg}\hspace{-2pt}\min_{x\in\mathcal{M}}h_\nu(x).
\end{equation}
For $h_\nu$ well chosen, $\widehat{x}^*$ can be obtained at significantly lower computational cost than $x^*$. A computable bound on $\frac{f_\nu(\widehat{x}^*) - f_\nu(x^*)}{f_\nu(x^*)}$ is given in \Cref{thm:functional_gap} and numerical comparisons are conducted in \cref{sec:experiments}.
\vspace{-0.2cm}
\section{Preliminaries on manifolds}\label{sec:preliminaries}
We briefly introduce essential notions of Riemannian geometry. Details can be found, e.g., in \cite{AbsMahSep2008,sakai1996riemannian}. A Riemannian metric $g$, \emph{metric} for short, is a family $\{g_x:T_x\mathcal{M}\times T_x\mathcal{M}\rightarrow \mathbb{R}\}_{x\in\mathcal{M}}$ of symmetric positive definite bilinear forms that depend smoothly on the location $x$. For all $\xi\in T_x\mathcal{M}$, the norm induced by the metric is $\|\xi\|_g\coloneq \sqrt{g_x(\xi,\xi)}$ and the length $l_g(\gamma)$ of every continuously differentiable curve $\gamma:[0,1]\rightarrow\mathcal{M}$ is given by $l_g(\gamma)=\int_0^1\left\|\frac{\mathrm{d}}{\mathrm{d}t}\gamma(t)\right\|_g\mathrm{d}t$. If $\mathcal{M}$ is a complete Riemannian manifold, it follows from the Hopf-Rinow theorem \cite[Chap.~III, Thm.~1.1]{sakai1996riemannian} that the distance function induced by $g$ satisfies
\begin{equation*}
	d_g(x,\widetilde{x})= \min\{l_g(\gamma)\ | \ \gamma:[0,1]\rightarrow\mathcal{M},\gamma(0)=x,\ \gamma(1)=\widetilde{x}\}.
\end{equation*}
A \emph{geodesic} $\gamma_g:\mathbb{R}\rightarrow\mathcal{M}$ is a locally minimizing curve with constant speed $v>0$, i.e., for all $t\in \mathbb{R}$, there is $\varepsilon>0$ such that for all $s\in\mathbb{R}$, if $|t-s|<\varepsilon$, then $d_g(\gamma_g(t),\gamma_g(s)) = |t-s|v$.

The Riemannian exponential at $x\in\mathcal{M}$ is the function $\mathrm{Exp}_x:T_x\mathcal{M}\rightarrow \mathcal{M}$ mapping $\xi$ to the point reached at unit time by the geodesic with starting point $x$ and initial velocity $\xi$, see, e.g., \cite[Chap.~II, Sec.~2]{sakai1996riemannian}. The Riemannian logarithm is defined in~\cite{mataigne2024bounds,mataigne2025}:
\begin{definition}\label{def:logarithm}
	Let $x,\widetilde{x}\in\mathcal{M}$ where $\mathcal{M}$ is a complete Riemannian manifold endowed with a metric $g$. The Riemannian logarithm $\mathrm{Log}_x(\widetilde{x})$ is a set-valued function returning all $\xi\in T_x\mathcal{M}$ such that
	\begin{equation}
		\mathrm{Exp}_x(\xi)=\widetilde{x} \text{ and } \|\xi\|_g = d_g(x,\widetilde{x}).
	\end{equation}
	The curves $[0,1]\ni t\mapsto\mathrm{Exp}_x(t\xi)$ are then called minimal geodesics.
\end{definition}
%\begin{minipage}{0.5\textwidth}

It is known that the subset of $\mathcal{M}$ where $\mathrm{Log}_x$ is not a singleton is included in the \emph{cut locus} $\mathcal{C}_x$ and is a zero measure set of $\mathcal{M}$~\cite[Chap.~III,~Lem.~4.4]{sakai1996riemannian}. In numerical algorithms, it is thus harmless to consider that $\mathrm{Log}_x$ returns a singleton. The Riemannian logarithm is key to the Riemannian $l^\nu$-barycenter because, if $x\notin \cup_{i=1}^N \mathcal{C}_{x_i}$(and $x\notin \{x_i\}_{i=1}^N$ if $\nu<2$), then $\mathrm{grad}_x \left(d_g(x,x_i)^\nu\right) =-\nu d_g(x,x_i)^{\nu-2}\mathrm{Log}_x(x_i)$~\cite[Eq.~2.8]{Afsari13}, where $\mathrm{grad}_x$ denotes the Riemannian gradient. If $\nu\geq 2$ and if $x^*\notin \cup_{i=1}^N \mathcal{C}_{x_i}$, it can be deduced that $\mathrm{grad}_x(f_{\nu }(x^*)^\nu)=0$ if and only if~\cite[Sec.~2.1.5]{Afsari13}
\begin{equation}\label{eq:stationary_point}
	\sum_{i=1}^N d_g(x^*,x_i)^{\nu-2}\mathrm{Log}_{x^*}(x_i)=0_x.
\end{equation} 
The stationarity condition \eqref{eq:stationary_point} is illustrated in \Cref{fig:karcher_mean}.

\begin{figure}[ht]
	\vspace{-0.2cm}
	\centering
	\includegraphics[width = 0.65\textwidth ]{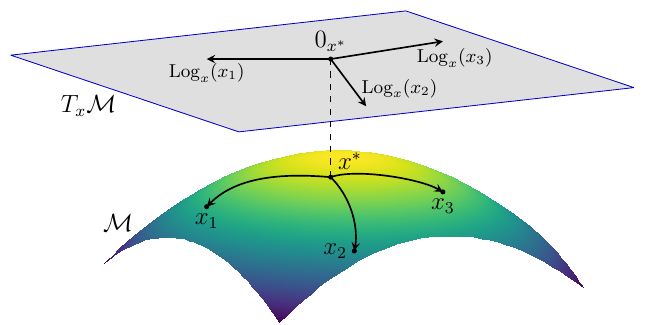}
	\vspace{-0.1cm}
	\caption{An artist view of the Riemannian $l^\nu$-barycenter. The figure represents a collection of points $\{x_i\}_{i=1}^3\subset\mathcal{M}$ and their Riemannian $l^\nu$-barycenter $x^*$. If $\nu\geq 2$, $x^*$ must be a first-order stationary point of \eqref{eq:karcher_mean}, i.e., satisfies \eqref{eq:stationary_point}.}
	\label{fig:karcher_mean}
\end{figure}
\vspace{-0.8cm}
\section{Minorizing and majorizing the objective function}
Assume one knows two functions $\widehat{m}_g,\widehat{M}_g:\mathcal{M}^2\rightarrow\mathbb{R}$ such that for all pairs $x, \widetilde{x}\in\mathcal{M}$, we have
\begin{equation}\label{eq:bounds}
	\widehat{m}_g(x,\widetilde{x})\leq d_g(x, \widetilde{x})\leq \widehat{M}_g(x,\widetilde{x}).
\end{equation}
If $\mathcal{M}$ is embedded in the Euclidean space $\mathbb{R}^M$, then $\widehat{m}_g$ and $\widehat{M}_g$ can often be chosen as functions of the Euclidean distance $\|x-\widetilde{x}\|_\mathrm{E}$. We give explicit examples of such functions on the Stiefel manifold $\mathrm{St}(n,p)$, $n\geq 2p$, endowed with the so-called $\beta$-metric \cite[Eq.~2.1]{mataigne2024bounds} in \cite[Thm.~7.1]{mataigne2024bounds}:
\begin{align*}
	\widehat{m}_\beta(X,\widetilde{X}) &= \min\{1,\sqrt{\beta}\}2\sqrt{p}\arcsin\left(\frac{\|X-\widetilde{X}\|_\mathrm{E}}{2\sqrt{p}}\right),\\
	\text{and}\quad \widehat{M}_\beta(X,\widetilde{X})&=\max\{1,\sqrt{\beta}\}\begin{cases}
	2\arcsin\left(\frac{\|X-\widetilde{X}\|_\mathrm{E}}{2}\right) \text{ if } \|X-\widetilde{X}\|_\mathrm{E}\leq 2,\\
	\frac{\pi}{2}\|X-\widetilde{X}\|_\mathrm{E}\hspace{1.1cm} \ \text{otherwise.}
	\end{cases}
\end{align*}
Let us define under- and over-approximations of $f_\nu$ from \eqref{eq:karcher_mean} by
\begin{equation*}
	\widehat{f}_\nu(x)\coloneq \left(\sum_{i=1}^N \widehat{m}_g(x,x_i)^\nu\right)^{\frac{1}{\nu}}\quad\text{and}\quad\widehat{F}_\nu(x)\coloneq \left(\sum_{i=1}^N \widehat{M}_g(x,x_i)^\nu\right)^{\frac{1}{\nu}}.
\end{equation*}
It is clear by~\eqref{eq:bounds} that for all $x\in\mathcal{M}$, we have $\widehat{f}_\nu(x)\leq f_\nu(x)\leq \widehat{F}_\nu(x)$. 
%\subsection{Bounds on the error}

A choice for $h_\nu$ in \eqref{eq:approx_mean} that offers theoretical guarantees is the minimization of $\widehat{f}_\nu$. Indeed, by setting $h_\nu\coloneq\widehat{f}_\nu$, it follows by the optimality properties of a Riemannian $l^\nu$-barycenter $x^*$ and its approximation $\widehat{x}^* \in \mathrm{arg}\hspace{-1.5pt}\min_{x\in\mathcal{M}}\widehat{f}_\nu(x)$ that the following sequence of inequalities holds:
\begin{equation}\label{eq:inequalities}
	\widehat{f}_\nu(\widehat{x}^*)\leq \widehat{f}_\nu(x^*)\leq f_\nu(x^*)\leq f_\nu(\widehat{x}^*)\leq \widehat{F}_\nu(\widehat{x}^*).
\end{equation}
By~\eqref{eq:inequalities}, we can bound from above the functional relative error between $f_\nu(x^*)$ and $f_\nu(\widehat{x}^*)$. This is shown in \Cref{thm:functional_gap}. 
\begin{theorem}\label{thm:functional_gap}
	Let $\{x_i\}_{i=1}^N\subset\mathcal{M}$, $f_\nu$ and $x^*$ be defined as in \eqref{eq:karcher_mean}. Moreover, let $\widehat{x}^* \in \mathrm{arg}\hspace{-1.5pt}\min_{x\in\mathcal{M}}\widehat{f}_\nu(x)$, then \begin{equation}\label{eq:functional_upper_bound}
	\frac{f_\nu(\widehat{x}^*) - f_\nu(x^*)}{f_\nu(x^*)}\leq \frac{\widehat{F}_\nu(\widehat{x}^*)-\widehat{f}_\nu(\widehat{x}^*)}{\widehat{f}_\nu(\widehat{x}^*)}.
\end{equation}
\end{theorem}
\begin{proof}
	By~\eqref{eq:inequalities}, we have $f_\nu(\widehat{x}^*) - f_\nu(x^*)\leq \widehat{F}_\nu(\widehat{x}^*)-\widehat{f}_\nu(\widehat{x}^*)$. Moreover, it also holds that $\widehat{f}_\nu(\widehat{x}^*)\leq f_\nu(x^*)$. This concludes the proof.
\end{proof}
The upper bound from \Cref{thm:functional_gap} only depends on $\widehat{f}_\nu,\widehat{F}_\nu$ and $\widehat{x}^*$, which can all be computed at low computational cost. Moreover, the upper bound \eqref{eq:functional_upper_bound} can be small in practice, as exemplified in \cref{sec:experiments}. 

For any function $h_\nu$ in \eqref{eq:approx_mean}, the distance  $d_g(x^*,\widehat{x}^*)$ between a barycenter~$x^*$ and its approximation $\widehat{x}^*$ can be bounded from above by the computable quantity $N^{-\frac{1}{\nu}}2\widehat{F}_\nu(\widehat{x}^*)$, as shown in~\Cref{thm:upper_bound}.
\begin{theorem}\label{thm:upper_bound}
	Let $\{x_i\}_{i=1}^N\subset\mathcal{M}$, $f_\nu$ and $x^*$ be defined as in \eqref{eq:karcher_mean}. Then, for all $x\in\mathcal{M}$, we have $d_g(x^*,x)\leq N^{-\frac{1}{\nu}}2 f_\nu(x)\leq N^{-\frac{1}{\nu}}2 \widehat{F}_\nu(x) $. 
\end{theorem}
\begin{proof}

By the triangle inequality, for $i=1,...,N$, we have $d_g(x^*,x)\leq d_g(x^*, x_i) + d_g(x_i, x)$. Therefore, by usual $\nu$-norm inequalities in vector spaces, it follows that
\begin{align}
	\nonumber
	 d_g(x^*,x)&\leq \frac{1}{N}\left(\sum_{i=1}^N d_g(x^*, x_i) + \sum_{i=1}^N d_g(x_i, x)\right)\\
	\nonumber
	&\leq \frac{1}{N}\left[N^{1-\frac{1}{\nu}}\left(\sum_{i=1}^N d_g(x^*,x_i)^\nu\right)^\frac{1}{\nu} + N^{1-\frac{1}{\nu}}\left(\sum_{i=1}^N d_g(x,x_i)^\nu\right)^\frac{1}{\nu}\right]\\
	\nonumber
	&= N^{-\frac{1}{\nu}} (f_\nu(x^*)+f_\nu(x)) \leq N^{-\frac{1}{\nu}} 2f_\nu(x).
\end{align}
To conclude, it is enough to acknowledge that $f_\nu(x)\leq \widehat{F}_\nu(x)$.
\end{proof}
In particular, \Cref{thm:upper_bound} yields $d_g(x^*,\widehat{x}^*)\leq N^{-\frac{1}{\nu}}2 \widehat{F}_\nu(\widehat{x}^*)$. Therefore, setting $h_\nu=\widehat{F}_\nu$ in \eqref{eq:approx_mean} and thus $\widehat{x}^* \in \mathrm{arg}\hspace{-1.5pt}\min_{x\in\mathcal{M}}\widehat{F}_\nu(x)$ minimizes the upper bound on $d_g(x^*,\widehat{x}^*)$ from  \Cref{thm:upper_bound}.

Similarly, the quantity $d_g(x^*,\widehat{x}^*)$ can be bounded from below. However, in comparison with \Cref{thm:upper_bound}, the lower bound from \Cref{thm:lower_bound} can not be obtained without knowing $x^*$.
\begin{theorem}\label{thm:lower_bound}
	Let $\{x_i\}_{i=1}^N\subset\mathcal{M}$, $f_\nu$ and $x^*$ be defined as in \eqref{eq:karcher_mean}. Then, for all $x\in\mathcal{M}$, we have $d_g(x^*,x)\geq N^{-1} (f_\nu(x)-f_\nu(x^*))$. 
\end{theorem}
\begin{proof}
By the reverse triangle inequality, for $i=1,...,N$, we have $d_g(x^*,x)\geq |d_g(x^*, x_i) - d_g(x_i, x)| $. Therefore, it follows that
\begin{align}
\nonumber	
	d_g(x^*,x)&\geq \frac{1}{N}\sum_{i=1}^N |d_g(x_i, x) - d_g(x^*, x_i)|\geq \frac{1}{N}\left(\sum_{i=1}^N |d_g(x_i, x) - d_g(x^*, x_i)|^\nu \right)^\frac{1}{\nu}\\
	\label{eq:reverse_triangle}
	&\geq \frac{1}{N}\left[\left(\sum_{i=1}^N |d_g(x, x_i)|^\nu \right)^\frac{1}{\nu}-\left(\sum_{i=1}^N |d_g(x^*, x_i)|^\nu \right)^\frac{1}{\nu}\right]\\
	\nonumber
	&=N^{-1}(f_\nu(x)-f_\nu(x^*)),
\end{align}
where \eqref{eq:reverse_triangle} holds by the reverse triangle inequality.
\end{proof}
In view of \Cref{thm:lower_bound}, we have $d_g(x^*,\widehat{x}^*)\geq N^{-1} (f_\nu(\widehat{x}^*)-f_\nu(x^*))$.

To conclude this section, we briefly discuss the case of linear bounds $\widehat{m}_g$ and $\widehat{M}_g$ in terms of the Euclidean distance, i.e., there are $l_M\geq l_m>0$ such that $
	\widehat{m}_g(x,\widetilde{x})=l_m\|x-\widetilde{x}\|_\mathrm{E} \text{ and }\widehat{M}_g(x,\widetilde{x})=l_M\|x-\widetilde{x}\|_\mathrm{E}$.
Then, by choosing $\widehat{x}^* \in \mathrm{arg}\hspace{-1.5pt}\min_{x\in\mathcal{M}}\widehat{f}_\nu(x)$, \Cref{thm:functional_gap} yields
\begin{equation*}
	\frac{f_\nu(\widehat{x}^*) - f_\nu(x^*)}{f_\nu(x^*)}\leq \frac{l_M-l_m}{l_m}=\frac{l_M}{l_m}-1.
\end{equation*}
For example, on the Stiefel manifold $\mathrm{St}(n,p)$, $n\geq 2p$, we have $l_m=\min\{1,\sqrt{\beta}\}$ and $l_M=\max\{1,\sqrt{\beta}\}\frac{\pi}{2}$ \cite[Cor.~7.2]{mataigne2024bounds}. Moreover, if $\nu=2$, it is easily shown that $\widehat{x}^*$ is a projection of $\overline{x}=\frac{1}{N}\sum_{i=1}^N x_i$ on $\mathcal{M}$. Indeed, we have
\begin{equation*}
	\widehat{x}^*\in\mathrm{arg}\hspace{-2pt}\min_{x\in\mathcal{M}} \sum_{i=1}^N\|x-x_i\|_\mathrm{E}^2
	%&=\mathrm{arg}\hspace{-2pt}\min_{x\in\mathcal{M}} \sum_{i=1}^N\langle x,x\rangle_\mathrm{E}-2\langle x_i,x\rangle_\mathrm{E}\\
	%&=\mathrm{arg}\hspace{-2pt}\min_{x\in\mathcal{M}} \langle x,x\rangle_\mathrm{E}-2\langle\frac{1}{N}\sum_{i=1}^N x_i,x\rangle_\mathrm{E}\\
	=\mathrm{arg}\hspace{-2pt}\min_{x\in\mathcal{M}} \|x-\overline{x}\|_\mathrm{E}^2\eqcolon \mathrm{Proj}_\mathcal{M}(\overline{x}).
\end{equation*}
Finally, $\mathrm{Proj}_\mathcal{M}(\overline{x})$ can be efficiently computed on many manifolds such as the Grassmann manifold, the Stiefel manifold and the manifold of fixed-rank matrices. On the Stiefel manifold, it is given by the orthogonal factor of the polar decomposition. %Indeed, it is well known that if $\overline{X} = U\Sigma V^T$ is a thin SVD, then it follows that
%\begin{equation}
%	X=UV^T=\mathrm{arg}\hspace{-12.5pt}\min_{X\in\mathrm{St}(n,p)} \|X-\overline{X}\|_\mathrm{E}= \mathrm{Proj}_{\mathrm{St}(n,p)}(\overline{X}).
%\end{equation}  
However, nonlinear bounds in \eqref{eq:bounds} usually offer better approximations of the true Riemannian $l^\nu$ barycenter $x^*$.
\vspace{-0.2cm}
\section{Lifting maps as approximations of the Riemannian logarithm}\label{sec:inverse_retraction}

Retractions offer a framework in Riemannian optimization to approximate the Riemannian exponential at low computational cost~\cite{AbsMahSep2008}. A retraction $R$ is a smooth family of mappings $\{R_x:T_x\mathcal{M}\rightarrow\mathcal{M}\}_{x\in\mathcal{M}}$ such that for all $x\in\mathcal{M},\ \xi \in T_x\mathcal{M}$, $t\geq 0$, the retraction approximates the Riemannnian exponential at first order: $R_x(t\xi) = \mathrm{Exp}_x(t\xi)+\mathrm{o}(t)$. 

Similarly, \emph{lifting maps}, also called tangent-bundle maps, approximate the Riemannian logarithm~\cite{Fiori2015,ZhuSato2020,BarberoLinan2023}. A lifting map $L$ is a smooth family of mappings $\{L_x:\mathcal{M}\rightarrow T_x\mathcal{M}\}_{x\in\mathcal{M}}$ such that for all $x\in\mathcal{M}, \xi \in T_x\mathcal{M}$ and $t\geq 0$ small enough, $L_x(\mathrm{Exp}_x(t\xi)) =t\xi+ \mathrm{o}(t)$. As a result,
\begin{equation}\label{eq:distance_approx}
	 d_g(x, \mathrm{Exp}_x(t\xi))=\|L_x(\mathrm{Exp}_x(t\xi))\|_g+\mathrm{o}(t).
\end{equation}
In a neighborhood $\mathcal{V}\subset \mathcal{M}$ of $x$, a lifting map can be obtained by taking the inverse of a retraction, i.e., $L_x\coloneq R_x^{-1}$.
An example of such a lifting map is given by the orthogonal projection on the tangent space~$T_x\mathcal{M}$: $
L_x(\widetilde{x})\coloneq\mathrm{Proj}_{T_x\mathcal{M}}(\widetilde{x}-x)$. This map is locally the inverse of the so-called \emph{orthographic retraction} \cite{Tetsuya2013}. 

To approximate the Riemannian $l^\nu$-barycenter, the approach from \cite{Tetsuya2013,Fiori2015,bouchard2025} is to consider the limit of the fixed-point iteration
\begin{equation}\label{eq:fixedpoint}
 x^{k+1} = R_{x^k}\left(\sum_{i=1}^N \nu \|L_{x^k}(x_i)\|_g^{\nu-2} L_{x^k}(x_i)\right)\text{ for } k=0,1,2....,
\end{equation}
which approximates the stationarity condition~\eqref{eq:stationary_point}. However, for  $d_g(x,\widetilde{x})$ large enough, $L_x(\widetilde{x})$ can be a very poor approximation of $\mathrm{Log}_x(\widetilde{x})$, as illustrated in \Cref{fig:karcher_inverse_retraction}. An additional caveat is the difficulty to define the mapping $L_x(\widetilde{x})$ for all $\widetilde{x}\in\mathcal{M}$. For example, on the Stiefel manifold, the inverses of the polar retraction, the QR retraction and the Cayley retraction have guaranteed existence only in a neighborhood of $x$ \cite{Tetsuya2013}.
\begin{figure}[h]
	\vspace{-0.4cm}
	\centering
	\includegraphics[width=0.65\textwidth]{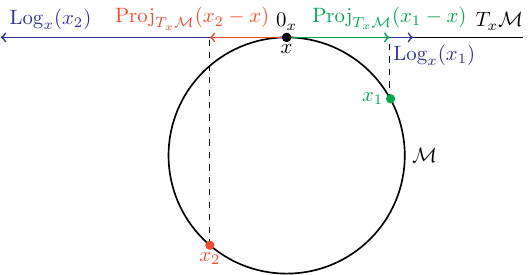}
	\vspace{-0.1cm}
	\caption{While the lifting map approximates well the Riemannian logarithm for points $x$ and $x_1$ that are close enough, they can differ significantly if the distance between the points $x$ and $x_2$ is large.}
	\label{fig:karcher_inverse_retraction}
\end{figure}
\section{Numerical experiments on the Stiefel manifold}\label{sec:experiments}
We conduct numerical experiments on the Stiefel manifold $\mathrm{St}(n,p)$ endowed with the $\beta$-metric \cite{mataigne2024bounds} for different values of $N,\nu,n,p$ and $\beta$. The code to reproduce these experiments is available at \url{https://github.com/smataigne/Approximate_barycenter.jl}. 

 We compare three algorithms: (A1) is the Riemannian gradient descent (RGD) on $(f_\nu)^\nu$, (A2) is the RGD on $(\widehat{f}_\nu)^\nu$ and (A3) is the fixed point iteration from~\eqref{eq:fixedpoint} with $R_X$ being the popular QR retraction and $L_X$ being the orthogonal projection on $T_X\mathrm{St}(n,p)$. The Riemannian logarithms in (A1) are numerically estimated using \cite[Alg.~2~and~4]{ZimmermannHuper22}. The sequences generated by (A1), (A2) and (A3) are respectively denoted by $\{X^k\}_{k=1}^{*}$, $\{Y^k\}_{k=1}^{*}$ and $\{Z^k\}_{k=1}^{*}$ with $X^0=Y^0=Z^0$.

The stopping criterion is as follows: (A1) runs until $|f_\nu(X^*)^\nu-f_\nu(X^{*-1})^\nu|\leq \varepsilon f_\nu(X^0)^\nu $, (A2) runs until $|\widehat{f}_\nu(Y^*)^\nu-\widehat{f}_\nu(Y^{*-1})^\nu|\leq \varepsilon  \widehat{f}_\nu(Y^0)^\nu $ and (A3) runs until $|g_\nu(Z^*)^\nu-g_\nu(Z^{*-1})^\nu|\leq \varepsilon  g_\nu(Z^0)^\nu $ where we define the function $g_\nu(Z)^\nu \coloneq \sum_{i=1}^N \|\mathrm{Proj}_{T_Z\mathrm{St}(n,p)}(X_i-Z)\|_\beta^\nu$. We choose $\varepsilon =10^{-6}$.

We consider spread and clustered datasets $\{X_i\}_{i=1}^N$: spread matrices are sampled with uniform distribution on $\mathrm{St}(n,p)$ while clustered matrices are sampled such that their pairwise distances do not exceed the known (when $\beta=1$~\cite{ZimmermannStoye25}) or suspected (when $\beta=\frac{1}{2}$~\cite{absil2025}) value of the injectivity radius.

The results of the experiments are given in \Cref{tab:table}, where they are compared with the bounds of \Cref{thm:functional_gap} and \Cref{thm:upper_bound}. In these experiments, we can observe that (i) both (A2) and (A3) provide quantitatively accurate approximations of the barycenter computed by (A1), (ii) the bounds are satisfied and (iii) the bounds give a better estimate of the error measures when the dataset is clustered.
\vspace{-0.2cm}
\begin{table}[h]
\centering
{\small 

    \begin{tabular}{||c|c|c|c|c||}
\hline
$f_\nu(X^*)$&$f_\nu(Y^*)$&$f_\nu(Z^*)$&$\frac{f_\nu(X^*)-f_\nu(Y^*)}{f_\nu(X^*)}\leq \frac{\widehat{F}_\nu(Y^*)-\widehat{f}_\nu(Y^*)}{\widehat{f}_\nu(Y^*)}$&$d_g(X^*,Y^*)\leq N^{-\frac{1}{\nu}}2\widehat{F}_\nu(Y^*)$\\
\hline
\multicolumn{5}{||c||}{Clustered dataset, $(N,n, p, \beta,\nu) = (5,50,20,1,1)$}\\
\hline
4.384 & 4.384 & 4.384 &    
 $1.114$e$-6 \leq 0.03344$ & 
 $0.001613\leq 1.811$  \\
\hline
\multicolumn{5}{||c||}{Spread dataset, $(N,n, p,\beta,\nu) = (5,50,20,1,2)$}\\
\hline
11.59 & 11.66 & 11.66 & 
 $0.005723 \leq 0.4886$ &
 $0.7046\leq 15.03$\\
\hline
\multicolumn{5}{||c||}{Clustered dataset, $(N,n, p,\beta,\nu) = (10,100,40,0.5,1)$}\\
\hline
8.426 & 8.427 & 8.427 & 
 $9.751$e$-5 \leq 0.4707$ &
 $0.01176\leq 1.953$\\
\hline
\multicolumn{5}{||c||}{Spread dataset, $(N,n, p, \beta,\nu) = (10,100,40,0.5,2)$}\\
\hline
23.50 & 24.23 & 24.23 & 
 $0.03116 \leq 1.076$ &
 $3.274\leq 23.52$\\
\hline
	\end{tabular}}
	\vspace{0.2cm}
	\caption{Average results over 10 runs of the performances of algorithms (A1), (A2) and (A3) for the computation of an approximate Riemannian $l^\nu$-barycenter.} 
	\label{tab:table}
	\vspace{-0.5cm}
\end{table}

In \Cref{fig:time}, we show convergence curves of algorithms (A1), (A2) and (A3) for $(N,n,p,\beta,\nu)=(5,100,40,\frac{1}{2}, 2)$ on a spread dataset. We also show the evolution of the running times as the size $p$ of $\mathrm{St}(140,p)$ increases. We can observe that the computational costs of (A2) and (A3) are significantly smaller than the cost of (A1) (by a factor larger than 100) and that (A2) and (A3) have similar computational costs. For large values of $N,n$ and $p$, (A1) is not tractable.

In conclusion, both (A2) and (A3) are computationally efficient approaches providing accurate approximations of the Riemannian $l^\nu$-barycenter. They are thus worthwhile alternatives to (A1). Moreover, (A2) benefits from theoretical guarantees (\Cref{thm:functional_gap}) for spread datasets where the behavior of (A3) is not known in general.

%{\color{blue} Ici, je vais ajouter un tableau qui compare les différentes méthodes, et les bornes de la section précédente. Il me restera même un peud e place pour motiver le problème au début si tout va bien : - )}

\begin{figure}[ht]
	\centering
	\vspace{-0.2cm}
	\includegraphics[width = 0.65\textwidth]{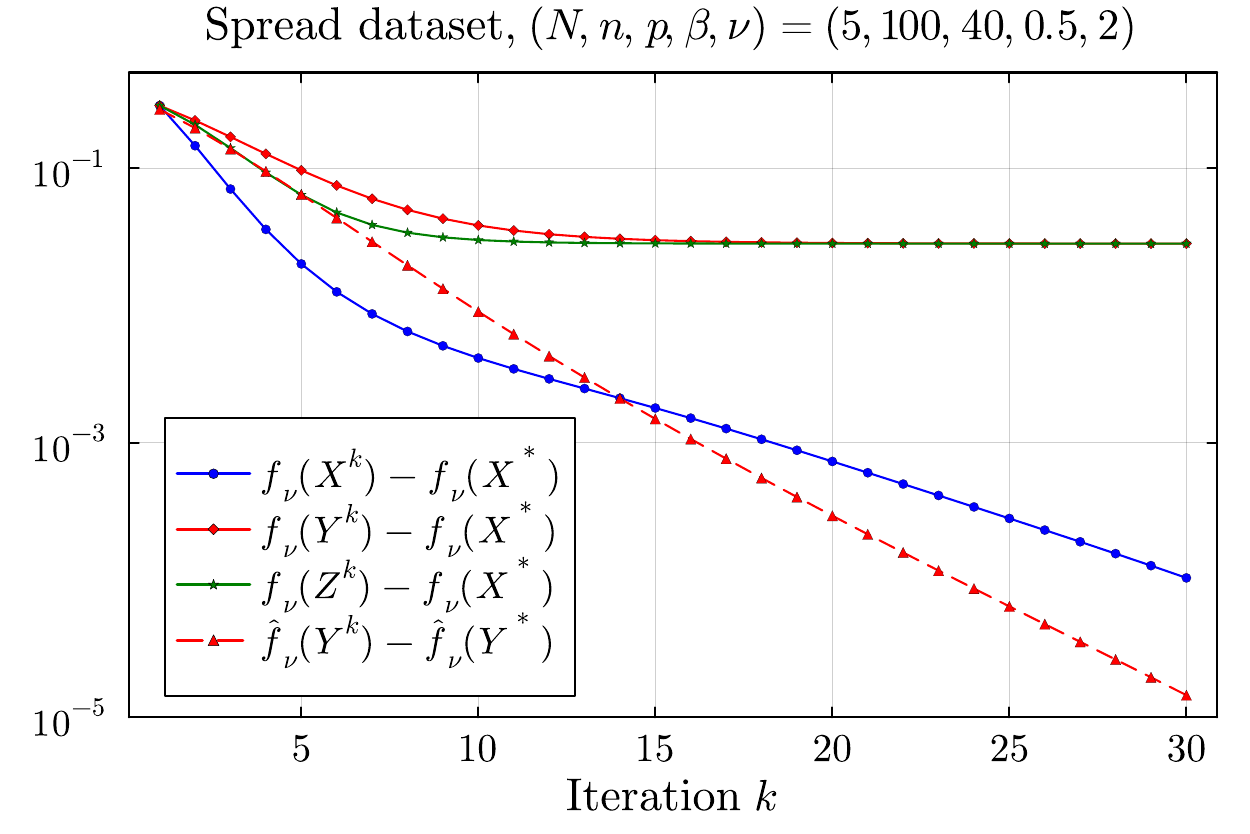}
		\includegraphics[width = 0.67\textwidth]{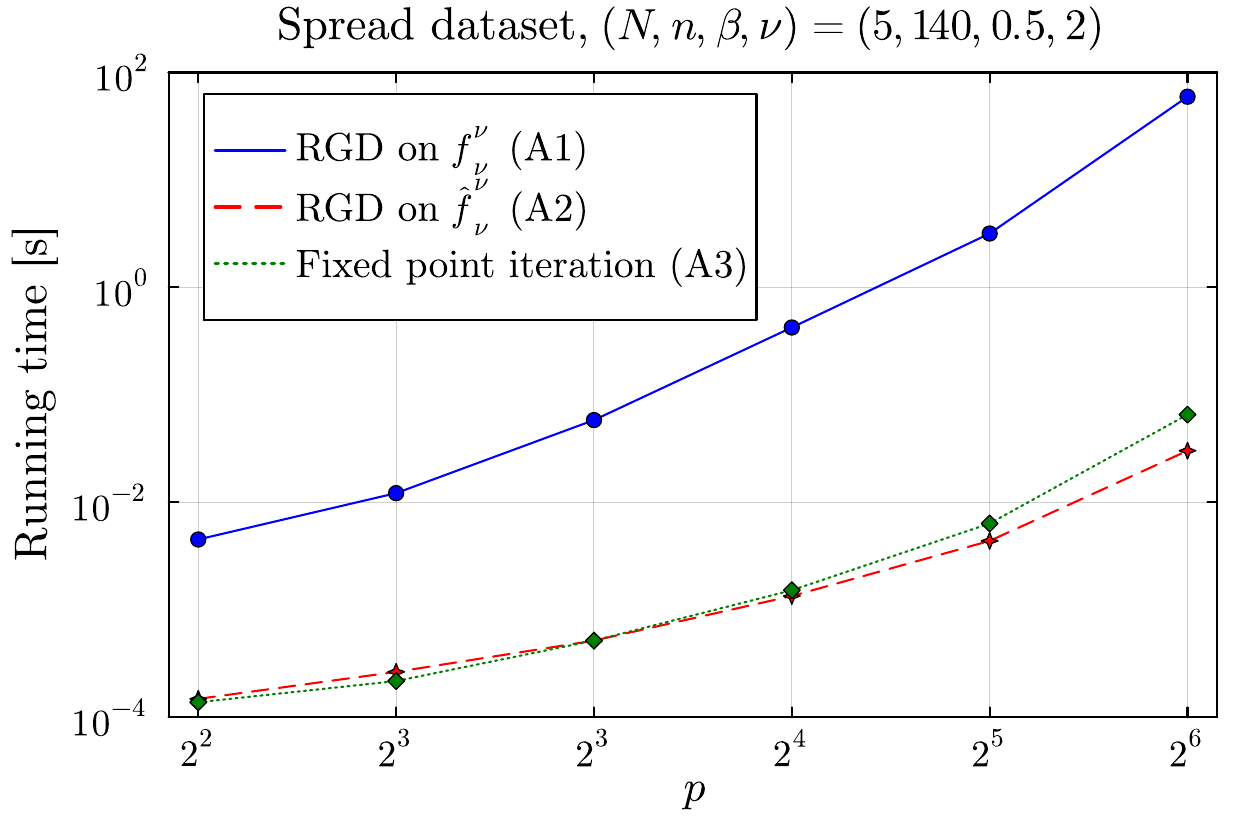}
		\vspace{-0.2cm}
		\caption{At the top: convergence curves of algorithms (A1), (A2) and (A3) to obtain the results of \Cref{tab:table}. For each algorithm, the evolution of $f_\nu$ is represented with solid lines. The dashed line represents, for $\{Y^k\}_{k=1}^{*}$, the evolution of the optimized objective function $ \widehat{f}_\nu$. At the bottom: a comparison of the evolution of the running time of the three algorithms as the size $p$ of $\mathrm{St}(140,p)$ increases using \href{https://juliaci.github.io/BenchmarkTools.jl/stable/}{BenchmarkTools.jl}.}
		\vspace{-0.6cm}
		\label{fig:time}
\end{figure}
\vspace{-0.4cm}
\begin{credits}
\subsubsection{\ackname} Simon Mataigne is a Research Fellow of the Fonds de la Recherche Scientifique - FNRS. This work was supported by the Fonds de la Recherche Scientifique - FNRS under grant no T.0001.23. Nina Miolane acknowledges funding from the NSF Career 2240158.

\subsubsection{\discintname}
The authors have no competing interests to declare that are
relevant to the content of this article.
%It is now necessary to declare any competing interests or to specifically
%state that the authors have no competing interests. Please place the
%statement with a bold run-in heading in small font size beneath the
%(optional) acknowledgments\footnote{If EquinOCS, our proceedings submission
%system, is used, then the disclaimer can be provided directly in the system.},
%for example: Or: Author A has received research grants from Company W. Author B has received a speaker honorarium from Company X and owns stock in Company Y. Author C is a member of committee Z.
\end{credits}
\bibliographystyle{splncs04}
\bibliography{approximatebib.bib}

\end{document}